\documentclass[a4paper,10pt]{article}
\usepackage{amsmath,amssymb,amsfonts,amsthm}
\usepackage{graphicx,color,subfigure,enumerate,multirow}
\usepackage[]{hyperref}
\usepackage{fullpage}

\newcommand\NN{\ensuremath{\mathbb{N}}}
\newcommand\eps{\ensuremath{\varepsilon}}

\newtheorem{thm}{Theorem}
\newtheorem{prop}[thm]{Proposition}

\newtheorem{lem}[thm]{Lemma}

\newtheorem{rmk}[thm]{Remark}

\title{On the parity of the number of nodal domains for an eigenfunction of the Laplacian on tori}
\author{Corentin L\'ena\footnote{Department of Mathematics \emph{Guiseppe Peano}, University of Turin, Via Carlo Alberto, 10, 10123 Turin, Italy 
\texttt{clena@unito.it}}}

\begin{document}
\maketitle
\begin{abstract}
In this note, we discuss  a question posed by T. Hoffmann-Ostenhof (see  \cite{Hof12}) concerning the parity of the number of nodal domains for a non-constant eigenfunction of the Laplacian on flat tori. We present two results. We first show that on the torus $(\mathbb{R}/2\pi\mathbb{Z})^{2}$, a non-constant eigenfunction has  an even number of nodal domains. We then consider the torus $(\mathbb{R}/2\pi\mathbb{Z})\times(\mathbb{R}/2\rho\pi\mathbb{Z})\,$, with $\rho=\frac{1}{\sqrt{3}}$\,, and construct on it an eigenfunction with three nodal domains.
 \end{abstract}




\section{Introduction}
\label{secIntro}

We consider the  non-negative Laplace-Beltrami operator $-\Delta$ on the torus $\mathbb{T}^2_{\rho}=(\mathbb{R}/2\pi\mathbb{Z})\times(\mathbb{R}/2\rho\pi\mathbb{Z})\,$, seen as a two-dimensional Riemannian manifold, with $\rho\in(0,1]\,$. The eigenvalues of $-\Delta$ are given by
\[\lambda_{m,n}=m^2+\frac{n^2}{\rho^2}\,,\]
with $(m,n)\in\NN^2\,$, and an associated basis of eigenfunctions is given, in the standard coordinates, by
\[u_{m,n}^{cc}(x_1,x_2)=\cos(mx_1)\cos\left(\frac{nx_2}{\rho}\right)\,;\]
	\[u_{m,n}^{cs}(x_1,x_2)=\cos(mx_1)\sin\left(\frac{nx_2}{\rho}\right)\,;\]
	\[u_{m,n}^{sc}(x_1,x_2)=\sin(mx_1)\cos\left(\frac{nx_2}{\rho}\right)\,;\]
	\[u_{m,n}^{ss}(x_1,x_2)=\sin(mx_1)\sin\left(\frac{nx_2}{\rho}\right)\,.\]
To be more precise, the family consisting of all the above functions that are non-zero is an orthogonal basis of $L^2(\mathbb{T}^2_{\rho})\,$. Let us note that the eigenspace associated with the eigenvalue $\lambda$ is spanned by all the functions in this basis such that the corresponding pair of indices $(m,n)$ satisfies $\lambda=m^2+\frac{n^2}{\rho^2}\,$. If $\rho^2$ is a rational number, a large eigenvalues can have a very high multiplicity, and an associated eigenfunction can possess a very complex nodal structure (see for instance \cite{BouRud11}).

We recall that for any eigenfunction $u$ of $-\Delta\,$, we call \emph{nodal set} the closed set $\mathcal{N}(u)=u^{-1}(\{0\})$ and \emph{nodal domain} a connected component of $\mathbb{T}_{\rho}^2\setminus \mathcal{N}(u)\,$. We will prove the following statements.

\begin{thm} 
\label{thmEven}
If $\rho^2$ is irrational or $\rho=1$, any non-constant eigenfunction $u$ of $-\Delta$ has an even number of nodal domains. More precisely, we can divide the nodal domains of $u$ into pairs of isometric domains, $u$ being positive on one domain of each pair and negative on the other.
\end{thm}

\begin{prop} 
\label{prop3Dom}
If $\rho=\frac{1}{\sqrt{3}}\,$, there exists an eigenfunction of $-\Delta$ with three nodal domains.
\end{prop}

In \cite{Hof12}, T. Hoffmann-Ostenhof asked if there exists a torus $(\mathbb{R}/2\pi\mathbb{Z})\times(\mathbb{R}/2\pi\rho\mathbb{Z})\,$, with $\rho \in (0,1]\,$, for which some eigenfunction of the Laplacian has an odd number of nodal domains, at least equal to three. Proposition \ref{prop3Dom} answers the question positively, while Theorem \ref{thmEven} shows that such an eigenfunction does not exist when $\rho^2$ is irrational or $\rho=1\,$.

\paragraph{Acknowledgements} I thank Bernard Helffer for introducing me to this problem and for numerous discussions and corrections. I thank Thomas Hoffmann-Ostenhof and Susanna Terracini for their advice and encouragements. This work was partially supported by the ANR (Agence Nationale de la Recherche), project OPTIFORM n$^\circ$
ANR-12-BS01-0007-02, and by the ERC, project COMPAT n$^\circ$ ERC-2013-ADG.

\section{Proof of the theorem}
\label{secProof}

Let us outline the method we will use to prove Theorem \ref{thmEven}. Let us first note that to any vector $v=(v_1,v_2)\in \mathbb{R}^2\,$, we can associate a bijection $x \mapsto x+v$ from $\mathbb{T}^2_{\rho}$ to itself . It is defined in the following way: if $x=(x_1,x_2)$ in the standard coordinates, $x+v=(x_1+v_1\mbox{ mod } 2\pi,x_2+v_2\mbox{ mod } 2\rho\pi)\,$. We will prove the following result. 
\begin{prop}
\label{propAS}
If $\rho^2$ is irrational or $\rho=1$, and if $u$ is a non-constant eigenfunction of $-\Delta\,$ on $\mathbb{T}^2_{\rho}\,$, there exists $v_u\in \mathbb{R}^2$ such that $u(x+v_u)=-u(x)$ for all $x\in \mathbb{T}_{\rho}^2\,$.
\end{prop}

Let us show that Proposition \ref{propAS} implies Theorem \ref{thmEven}. An eigenfunction $u$ being given, we define the bijection $\sigma: x\mapsto x+v_u$ from $\mathbb{T}_{\rho}^2$ to itself. It is an isometry that preserves $\mathcal{N}(u)\,$, and exchanges the nodal domains on which $u$ is positive with those on which $u$ is negative. This proves Theorem \ref{thmEven}.  

Let us now turn to the proof of Proposition \ref{propAS}. Let us first consider the case where $\rho^2$ is irrational, and let $\lambda$ be a non-zero eigenvalue of $-\Delta$. Since $\rho^2$ is irrational there exists a unique pair of integers $(m,n)$, different from $(0,0)\,$, such that $\lambda=m^2+\frac{n^2}{\rho^2}\,$. The eigenspace associated with $\lambda$ is therefore spanned by the functions $u^{cc}_{m,n}\,$, $u^{cs}_{m,n}\,$, $u^{sc}_{m,n}\,$, and $u^{ss}_{m,n}\,$. Let us assume that $m>0$ and let us set $v=(\pi/m,0)$. It is then immediate to check that, for all $x$ in $\mathbb{T}^2_{\rho}$, $u(x+v)=-u(x)$ when $u$ is any of the basis functions $u^{cc}_{m,n}\,$, $u^{cs}_{m,n}\,$, $u^{sc}_{m,n}\,$, and $u^{ss}_{m,n}\,$. As a consequence we still have $u(x+v)=-u(x)$ when $u$ is any linear combination of the previous basis functions, that is to say any eigenfunction associated with $\lambda$\,. If $m=0$, we have $n>0$ and the same holds true with $v=(0,\rho\pi/n)\,$. This conclude the proof of Proposition \ref{propAS} in the irrational case. 

Let us now consider the case $\rho=1$. As in the previous case, we will prove a statement that is slightly more precise than Proposition \ref{propAS}: we will exhibit, for any non-zero eigenvalue $\lambda$, a vector $v\in \mathbb{R}^2$ such that $u(x+v)=-u(x)$ for every eigenfunction $u$ associated with $\lambda$ (see Lemma \ref{lemEigFunct}). The difference in this case is that the equality $\lambda=m^2+n^2$ can be satisfied for several pairs of integers $(m,n)$. To overcome this difficulty, we will need the following simple arithmetical lemma. This result is stated and proved in \cite{Hof15}, where it is used to solve a closely related problem: proving that a non-constant eigenfunction of the Laplacian on the square with a Neumann or a periodic boundary condition must take the value $0$ on the boundary. We nevertheless give a proof of the lemma here for the sake of completeness.

\begin{lem}
\label{lemArithm}
	Let $(m,n)$ be a pair of non-negative integers, with $(m,n)\neq (0,0)\,$, and let us write $\lambda=m^2+n^2\,$. If $\lambda=2^{2p}(2q+1)$ with $(p,q)\in\mathbb{N}^2\,$, then $m=2^pm_0$ and $n=2^pn_0\,$, where exactly one of the integers $m_0$ and $n_0$ is odd. If on the other hand $\lambda=2^{2p+1}(2q+1)$ with $(p,q)\in\mathbb{N}^2\,$, then $n=2^pm_0$ and $n=2^pn_0\,$, where both integers $m_0$ and $n_0$ are odd.
\end{lem}

\begin{proof}
From the decomposition into prime factors, we deduce that we can write any positive integer $N$ as $N=2^{t}N_1\,$, with $t$ a non-negative and $N_1$ an odd integer. Let us first consider the case where $m$ or $n$ is zero. Without loss of generality, we can assume that $n=0\,$. We write $m=2^{r}m_1\,$. We are in the case $\lambda=2^{2p}(2q+1)$ with $p=r$ and $2q+1=m_1^2\,$, and we obtain the desired result by setting $m_0=m_1$ (odd) and $n_0=0$ (even). We now assume that both $m$ and $n$ are positive. We write $m=2^r m_1$ and $n=2^s n_1$ with $m_1$ and $n_1$ odd integers. Without loss of generality, we can assume that $r \le s\,$. We find $\lambda = 2^{2r}(m_1^2+2^{2(s-r)}n_1^2)\,$. If $r<s\,$, then $m_1^2+2^{2(s-r)}n_1^2$ is an odd integer, and we have $\lambda=2^{2p}(2q+1)\,$, with $p=r$ and $2q+1=m_1^2+2^{2(s-r)}n_1^2\,$. In that case, we set $m_0=m_1$ (odd) and $n_0=2^{s-r}n_1$ (even). If $r=s\,$, we find 
$\lambda=2^{2r}(m_1^2+n_1^2)\,$. We have furthermore $m_1=2m_2+1$ and $n_1=2n_2+1$, and therefore $m_1^2+n_1^2= 4(m_2^2+n_2^2+m_2+n_2)+2\,$ . We have $\lambda=2^{2p+1}(2q+1)$ with $p=r$ and $q=m_2^2+n_2^2+m_2+n_2\,$, and we set $m_0=m_1$ and $n_0=n_1\,$. 
\end{proof}

\begin{lem}
\label{lemEigFunct}
 Let $\lambda$ be a non-zero eigenvalue of $-\Delta\,$ on $\mathbb{T}^2_1$\,.
\begin{enumerate}[i.]
	\item If $\lambda=2^{2p}(2q+1)\,$, we set $v=(\pi/2^p,\pi/2^p)\,$, and we have $u(x+v)=-u(x)$ for every eigenfunction $u$ associated with $\lambda\,$.
	\item If $\lambda=2^{2p+1}(2q+1)\,$, we set $v=(\pi/2^p,0)\,$, and we have $u(x+v)=-u(x)$ for every eigenfunction $u$ associated with $\lambda\,$.
\end{enumerate}
\end{lem}
\begin{proof} Let us first consider the case where $\lambda=2^{2p}(2q+1)\,$. Let us choose a pair of indices $(m,n)$ such that $\lambda=m^2+n^2\,$, and let us consider one of the associated basis functions given in the introduction, say $u^{cc}_{m,n}(x,y)=\cos(mx)\cos(ny)$ to fix the ideas. According to Lemma \ref{lemArithm}, we have $m=2^{p}m_0$ and $n=2^p n_0\,$, where exactly one of the integers $m_0$ and $n_0$ is odd. We can assume, without loss of generality, that $m_0$ is odd and $n_0$ even. Then
\[\cos\left(m\left(x_1+\frac{\pi}{2^p}\right)\right)=\cos(mx_1+m_0 \pi)=-\cos(mx_1)\,,\]
\[\cos\left(n\left(x_2+\frac{\pi}{2^p}\right)\right)=\cos(mx_2+n_0 \pi)=\cos(nx_2)\,,\]
and therefore
\[u^{cc}_{m,n}\left(x+v\right)=-u_{m,n}^{cc}(x)\,.\]
We show in the same way that $u^{cs}_{m,n}(x+v)=-u^{cs}_{m,n}(x)\,$, $u^{sc}_{m,n}(x+v)=-u^{sc}_{m,n}(x)\,$, and $u^{ss}_{m,n}(x+v)=-u^{ss}_{m,n}(x)\,$. Since $v$ depends only on $\lambda\,$, we have $u(x+v)=-u(x)$ as soon as $u$ is a basis function associated with $\lambda\,$, and therefore, by linear combination, as soon as $u$ is an eigenfunction associated with $\lambda\,$.

The case $\lambda=2^{2p+1}(2q+1)$ can be treated in the same way, taking $v=(\pi/2^p,0)$ ($v=(0,\pi/2^p)$ would also be suitable). \end{proof}

\begin{rmk} \label{rmkEven} It can also be shown that Lemma \ref{lemArithm} still holds if we replace the equation
$\lambda=m^2+n^2$ by 
$\lambda=\alpha m^2+ \beta n^2\,$, where $\alpha$ and $\beta$ are odd integers such that $\alpha+\beta=2 \mbox{ mod }4\,$. This implies that the conclusion of Theorem \ref{thmEven} still holds if $\rho =\sqrt{\frac{\alpha}{\beta}}$, with $\alpha$ and $\beta$ as above.
\end{rmk}

\section{Proof of the proposition}
\label{secProofProp}

In this section, we assume that $\rho=\frac{1}{\sqrt{3}}\,$. Let us outline the idea we will use to construct the eigenfunction whose existence is asserted in Proposition \ref{prop3Dom}. It will belong to the eigenspace associated with the eigenvalue $4\,$. We start from the eigenfunction 
\[u_{1,1}^{cc}(x_1,x_2)=\cos(x_1)\cos\left(\sqrt{3}x_2\right)\,,\]
which has four rectangular nodal domains, shown in Figure \ref{fig1a}. 
\begin{figure}
	\begin{center}
		\subfigure[$\mathcal{N}(u_{1,1}^{cc})$\label{fig1a}]{\includegraphics[width=4cm]{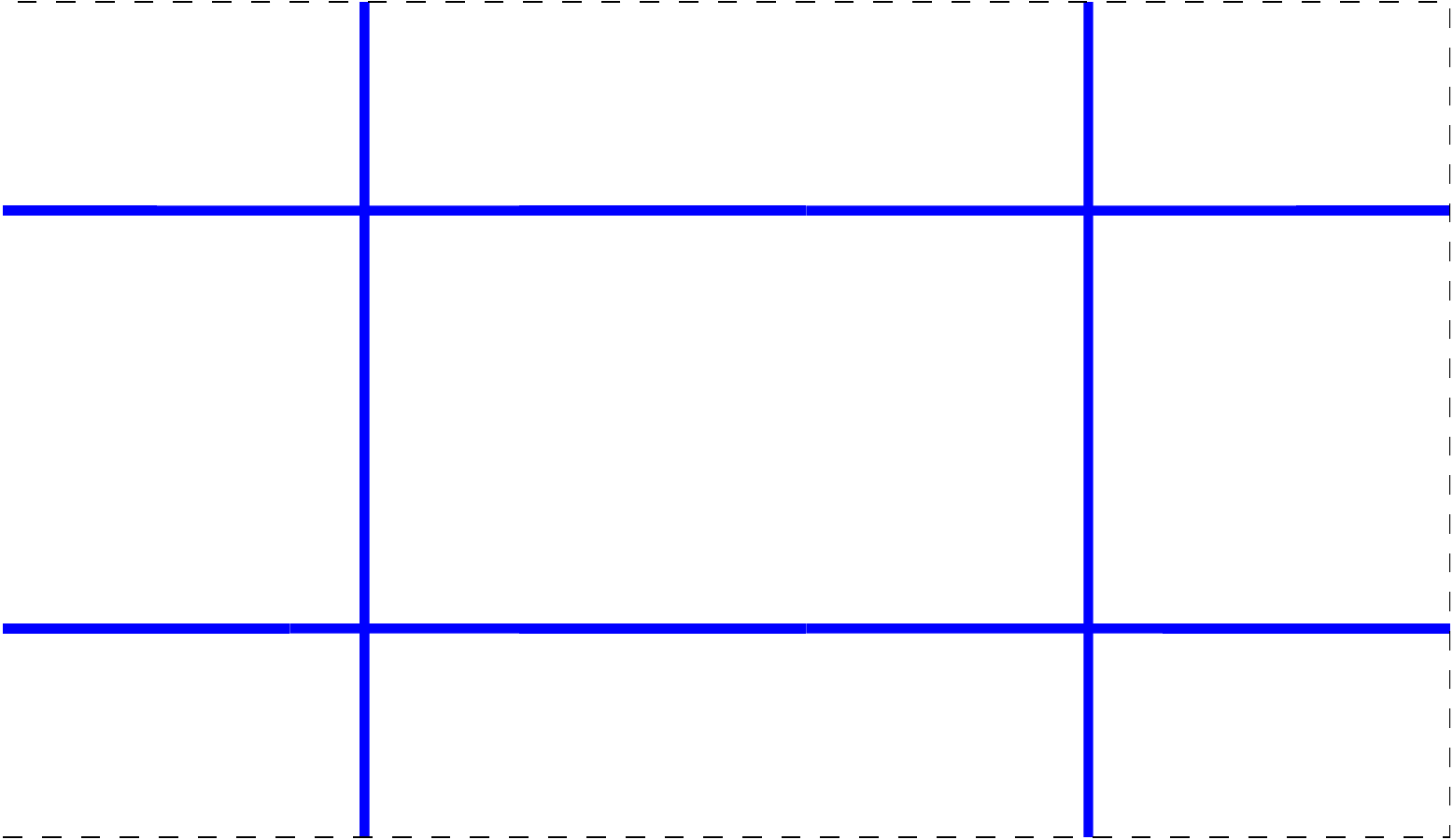}}
		\hspace{2cm}
		\subfigure[$\mathcal{N}(u_{2,0}^{cc})$\label{fig1b}]{\includegraphics[width=4cm]{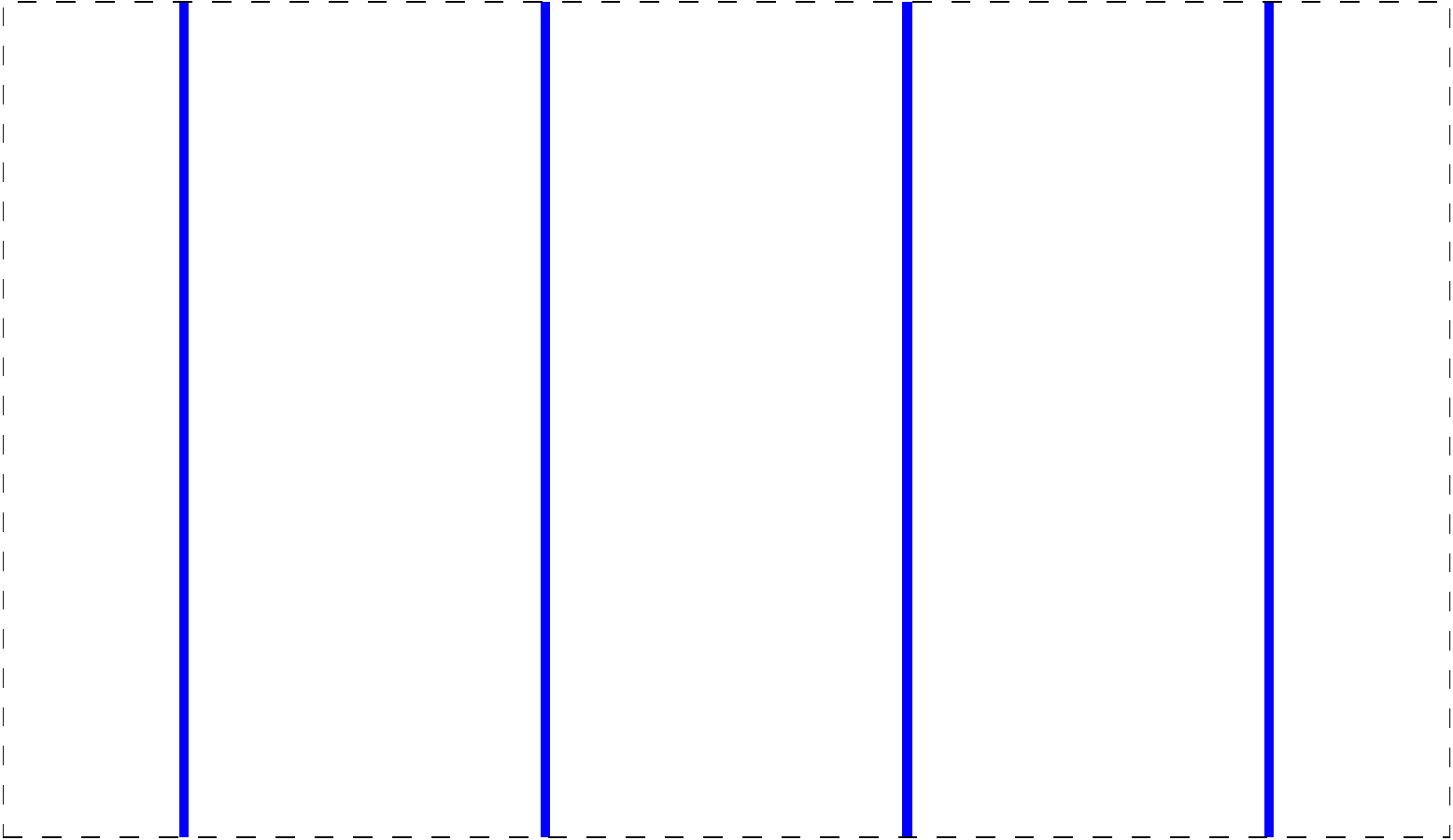}}
		\caption{Nodal sets of basis functions\label{fig1}}
	\end{center}
\end{figure}
We perturb this eigenfunction by adding a small multiple of the eigenfunction 
\[u_{2,0}^{cc}(x_1,x_2)=\cos(2x_1)\,.\]
For $\eps>0\,$, we get the eigenfunction
\[v_{\eps}(x_1,x_2)=u_{1,1}^{cc}(x_1,x_2)+\eps u_{2,0}^{cc}(x_1,x_2)\,.\]
Since $u_{2,0}^{cc}\,$ is negative in the neighborhood of the critical points in  $\mathcal{N}(u_{1,1}^{cc})$\,, adding $\eps u_{2,0}^{cc}$ has the effect of opening small "channels" that connect the nodal domains where $u_{1,1}^{cc}$ is negative. As a result, if  $\eps>0$ is small enough, $v_{\eps}$ has three nodal domains, one where it is negative and two where it is positive (see Figure \ref{fig2}). Let us note that these ideas have already been used, to construct examples of eigenfunctions whose nodal set satisfies some prescribed properties, in \cite{Ste25,Lew77,BerHel15Square}. In particular, the desingularization of critical points in the nodal set, that we have briefly described, is studied in details in \cite[6.7]{BerHel15Square}
\begin{figure}
	\begin{center}
		\includegraphics[width=6cm]{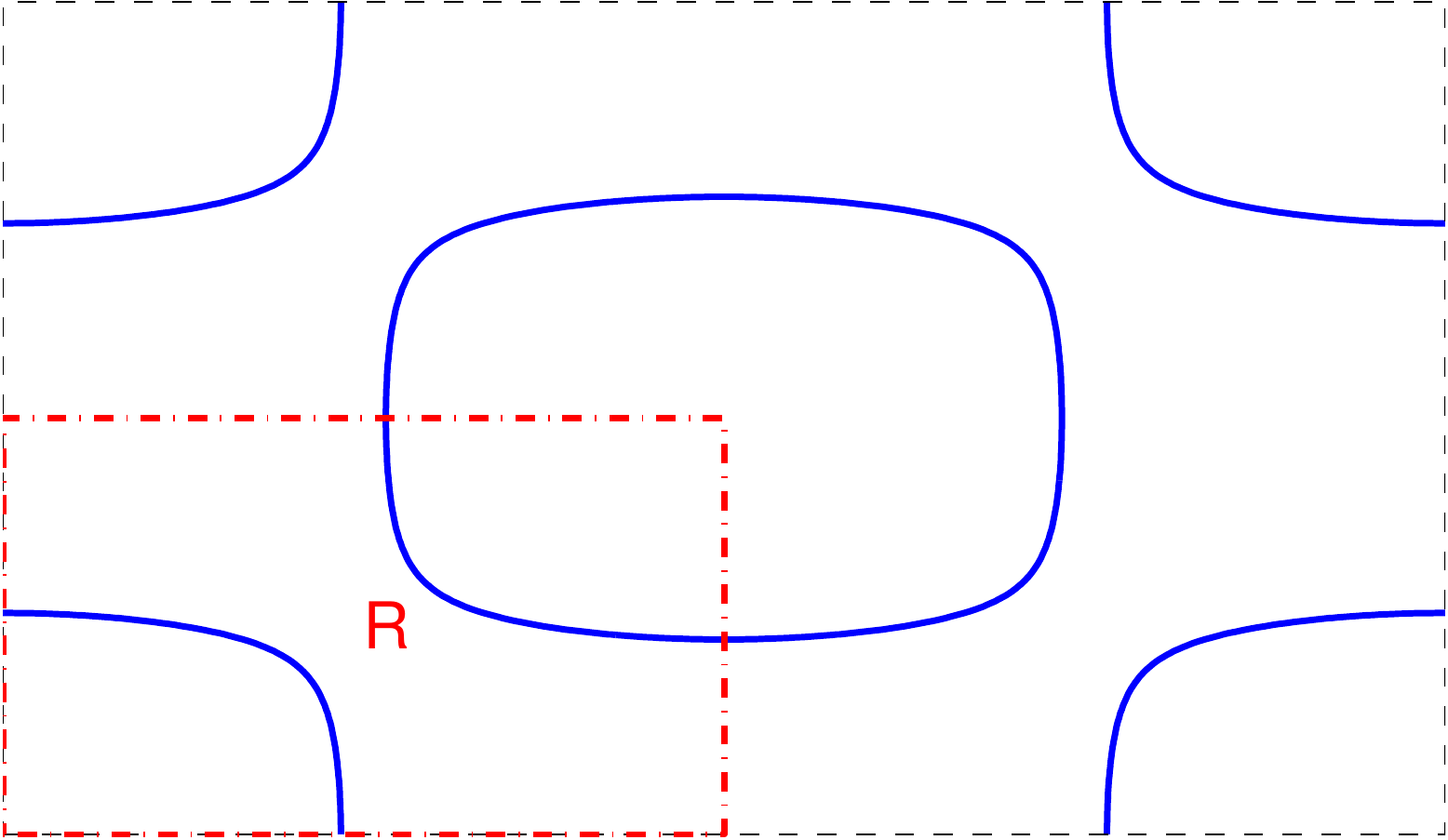}
		\caption{Nodal set of $v_{\eps}=u_{1,1}^{cc}+\eps u_{2,0}^{cc}$ with $\eps=0.1$\label{fig2}}
	\end{center}
\end{figure}

To prove rigorously these assertions, let us consider the open domain $R$ in $\mathbb{T}^2_{\rho}$ defined, in the standard coordinates $(x_1,x_2)$, as
\[R=]0,\pi[\times\left]0,\frac{\pi}{\sqrt{3}}\right[\,.\]
We now define (following \cite{BerHel15Square}) the smooth change of coordinates
\[\left\{\begin{array}{lcl}
 \xi_1&=&-\cos(x_1),\\
 \xi_2&=&-\cos(x_2),\\
 		\end{array}
 	\right.\]
which sends $R$ to $]-1,1[\times]-1,1[$\,.
In these new coordinates the nodal set of the function $v_{\eps}$ satisfies the equation
\[\xi_1\xi_2+\eps(2\xi_1^2-1)=0\,.\]
A simple computation shows that this is the equation of an hyperbola. Furthermore, when $0<\eps<1$, this hyperbola has one branch in the lower left quadrant $]-1,0[\times]-1,0[$ and one in the upper right quadrant $]0,1[\times]0,1[\,$.

On the other hand, we have the following symmetries of $v_{\eps}$:
 \[v_{\eps}(2\pi-x_1,x_2)=v_{\eps}(x_1,x_2)\]
 and
 \[v_{\eps}(x_1,2\pi-x_2)=v_{\eps}(x_1,x_2).\]
 This allows us to recover the nodal set on the whole of $\mathbb{T}^2_{\rho}\,$. It consists in two simple closed curves, each containing one branch of the previously considered hyperbola. Each of these closed curves enclose a nodal domain that is homeomorphic to a disk.  Let us now consider the complement of the closure of those two region. It is the third nodal domain.
 
\begin{rmk}
\label{rmkOdd}
 The construction used to prove Proposition \ref{prop3Dom} can obviously be generalized. We can for instance consider the eigenfunction 
\[v_{\eps}(x_1,x_2)=u^{cc}_{m,n}(x_1,x_2)+\eps u^{c,c}_{km,0}(x_1,x_2)\,,\]
assuming that 
\[m^2+\frac{n^2}{\rho^2}=k^2m^2\,.\]
We have in that case $\rho=\frac{n}{m\sqrt{k^2-1}}\,$. Following the same line of reasoning as in this section, we see that for $\eps>0$ small enough, $v_{\eps}$ has $2mn+1$ nodal domains.  
\end{rmk} 

In view of Remarks \ref{rmkEven} and \ref{rmkOdd}, it would be desirable to obtain a characterization of the rational numbers $q\,$, such that there exists an eigenfunction of $-\Delta$ on the torus $\mathbb{T}_{\sqrt{q}}^2$ with an odd number of nodal domains. Unfortunately, we have not been able to reach this goal so far.

\bibliographystyle{plain}

{\small
\begin{thebibliography}{1}

\bibitem{BerHel15Square}
P.~{B{\'e}rard} and B.~{Helffer}.
\newblock {A. Stern's analysis of the nodal sets of some families of spherical
  harmonics revisited}.
\newblock In S.~Kallel, N.~Mir, El~Kacimi, and A.~A., Baklouti, editors, {\em
  Analysis and Geometry, MIMS-GGTM, Tunis, Tunisia, March 2014. In Honour of
  Mohammed Salah Baouendi}, volume 127 of {\em Springer Proceedings in
  Mathematics {\&} Statistics}. Springer, 2015.
\newblock To appear.

\bibitem{BouRud11}
J.~Bourgain and Z.~Rudnick.
\newblock On the geometry of the nodal lines of eigenfunctions of the
  two-dimensional torus.
\newblock {\em Annales Henri Poincar{\'e}}, 12(6):1027--1053, 2011.

\bibitem{Hof12}
T.~{Hoffmann-Ostenhof}.
\newblock Geometric aspects of spectral theory, {P}roblem {S}ection (xv).
\newblock {\em Oberwolfach Rep.}, 9(3):2013--2076, 2012.

\bibitem{Hof15}
T.~{Hoffmann-Ostenhof}.
\newblock Eigenfunctions for 2-dimensional tori and for rectangles with neumann
  boundary conditions.
\newblock {\em Mosc. Math. J.}, 15(1):101--106, jan.--mar. 2015.

\bibitem{Lew77}
Hans Lewy.
\newblock On the minimum number of domains in which the nodal lines of
  spherical harmonics divide the sphere.
\newblock {\em Comm. Partial Differential Equations}, 2(12):1233--1244, 1977.

\bibitem{Ste25}
{Stern, A.}
\newblock {\em Bemerkungen {\"u}ber asymptotisches Verhalten von Eigenwerten
  und Eigenfunctionen.}
\newblock PhD thesis, {Universit{\"a}t G{\"o}ttingen}, 1925.
\newblock {Extracts and annotations available at
  \url{http://www-fourier.ujf-grenoble.fr/~pberard/R/stern-1925-thesis-partial-reprod.pdf}}.

\end{thebibliography}
}

\end{document}